\documentclass[11pt,a4paper]{article}
\usepackage[margin=25mm]{geometry}
\usepackage[T1]{fontenc}
\usepackage[utf8]{inputenc}
\usepackage{lmodern}
\usepackage{authblk}
\usepackage{amsmath,amssymb,amsthm,mathtools}
\usepackage{microtype,enumitem,xcolor,booktabs,array}
\usepackage[font=small,labelfont=bf]{caption}
\usepackage{tikz}
\usetikzlibrary{positioning,fit,calc,arrows.meta,backgrounds,decorations.pathreplacing}
\usepackage[colorlinks=true,linkcolor=blue!50!black,citecolor=blue!50!black,urlcolor=blue!50!black]{hyperref}
\setlist[itemize]{leftmargin=1.6em,itemsep=2pt,topsep=3pt}
\setlist[enumerate]{leftmargin=2.1em,itemsep=3pt,topsep=3pt}
\theoremstyle{plain}
\newtheorem{theorem}{Theorem}[section]
\newtheorem{lemma}[theorem]{Lemma}
\newtheorem{proposition}[theorem]{Proposition}
\newtheorem{corollary}[theorem]{Corollary}
\theoremstyle{definition}
\newtheorem{definition}[theorem]{Definition}

\newcommand{\ch}{\operatorname{\chi_{\ell}}}
\newcommand{\sur}{\sigma}                 
\newcommand{\NX}{N_X}
\newcommand{\NY}{N_Y}
\newcommand{\FX}{F_X}
\newcommand{\FY}{F_Y}
\newcommand{\LH}{L(H)}
\newcommand{\Kminus}[3]{K(#1,#2)-#3}      
\newcommand{\eX}{e_X}
\newcommand{\eY}{e_Y}
\newcommand{\pX}{p_X}
\newcommand{\pY}{p_Y}
\newcommand{\stepname}[1]{\par\medskip\noindent\textbf{#1}\par\nopagebreak\smallskip}
\definecolor{Xcolour}{RGB}{20,112,150}
\definecolor{Ycolour}{RGB}{190,112,28}
\usetikzlibrary{arrows.meta,backgrounds,shapes.geometric}
\tikzset{
  vtx/.style={circle,draw,fill=white,inner sep=0pt,minimum size=5.5pt},
  vtxT/.style={vtx,fill=orange!70},
  vtxz/.style={vtx,fill=black},
  vtxdel/.style={circle,draw=gray!60,fill=gray!15,inner sep=0pt,minimum size=5.5pt},
  clique/.style={draw,rounded corners=7pt,fill=gray!12,inner sep=6pt},
  complete/.style={line width=3.2pt,gray!55,shorten <=1.5pt,shorten >=1.5pt},
  missing/.style={dashed,line width=0.9pt,red!75!black},
  flat/.style={line width=1.4pt,blue!60!black},
  lbl/.style={font=\small},
  slbl/.style={font=\scriptsize},
  box/.style={draw,rounded corners=3pt,align=center,font=\small,inner sep=4pt,text width=54mm},
  dec/.style={draw,rounded corners=3pt,align=center,font=\small,inner sep=4pt,fill=blue!6,text width=42mm},
  fin/.style={draw,rounded corners=3pt,align=center,font=\small\bfseries,fill=gray!15,inner sep=5pt},
  arr/.style={-{Stealth[length=2mm]},thick},
}
\hypersetup{pdftitle={Elementary graphs are chromatic-choosable},
            pdfsubject={List colouring of elementary graphs}}
\title{\bfseries Elementary graphs are chromatic-choosable}
\author[1]{Narayanan N}
\author[2]{Nandana K Vasudevan}
\author[2]{K Somasundaram} 
\affil[1]{Department of Mathematics \\
Indian Institute of Technology Madras, Chennai, India.}
\affil[2]{Department of Mathematics\\ Amrita School of Physical Sciences - Coimbatore \\ Amrita Vishwa Vidyapeetham, India.}
\date{}
\begin{document}
\maketitle
\thispagestyle{plain}
\begin{abstract}
A graph is \emph{elementary} if its edges can be coloured with two colours so that the two
edges of every induced path on three vertices receive different colours. Elementary graphs are
one of the two basic classes in the Chv\'atal--Sbihi decomposition of claw-free perfect graphs;
they include all line graphs of bipartite multigraphs and all complements of bipartite graphs.
We prove that every elementary graph $G$ satisfies $\ch(G)=\chi(G)=\omega(G)$, confirming for this
class the conjecture of Gravier and Maffray that claw-free graphs are chromatic-choosable.
The proof combines the Maffray--Reed description of elementary graphs as augmentations of line
graphs of bipartite multigraphs, Galvin's theorem, an extremal (edge-maximal) choice of
counterexample, and a Hall-type lemma describing how list colourings of a complete graph minus
a matching can fail.
\end{abstract}
\section{Introduction}\label{sec:intro}
All graphs in this paper are finite. Graphs whose vertices are being coloured are simple; the
multigraphs whose line graphs appear in the proof may have parallel edges but no loops.
A \emph{list assignment} on a graph $G$ gives every vertex $v$ a finite set $L(v)$ of colours,
and an \emph{$L$-colouring} is a proper colouring $f$ with $f(v)\in L(v)$ for all $v$. The graph
is \emph{$k$-choosable} if it has an $L$-colouring whenever all lists have size at least $k$, and
its \emph{choice number} $\ch(G)$ is the least such $k$
\cite{ERT80,Viz76}. Since ordinary colourings are list colourings with identical lists,
$\omega(G)\le\chi(G)\le\ch(G)$, where $\omega$ and $\chi$ are the clique number and the chromatic
number. A graph with $\ch(G)=\chi(G)$ is called \emph{chromatic-choosable}. The concept was introduced independentry by Vizing \cite{Viz76} and Erd\H{o}s, Rubin and Taylor \cite{ERT80}. We refer the interested reader to the dynamic survey by Vasudevan et al. \cite{NSN} for more details on the literature.
Following Chv\'atal and Sbihi \cite{CS88}, a graph is \emph{elementary} if its edges can be
coloured red and blue so that every induced path on three vertices has one red and one blue
edge. Chv\'atal and Sbihi showed that every claw-free perfect graph can be decomposed along
clique cutsets into elementary graphs and so-called peculiar graphs, and Maffray and Reed
\cite{MR99} later described elementary graphs exactly: they are the graphs obtained from line
graphs of bipartite multigraphs by replacing the two ends of some pairwise disjoint ``flat''
edges by complements of bipartite graphs. In particular, line
graphs of bipartite multigraphs and complements of bipartite graphs are elementary. We prove the following:
\begin{theorem}\label{thm:main}
Every elementary graph $G$ satisfies $
\ch(G)=\chi(G)=\omega(G).
$
\end{theorem}
The List Colouring Conjecture, formulated independently by several authors (see
\cite{Gal95,GM98} for its history), asserts that every line graph is chromatic-choosable.
Galvin \cite{Gal95} proved it for line graphs of bipartite multigraphs. Gravier and Maffray
\cite{GM98} asked whether every \emph{claw-free} graph is chromatic-choosable, proved this for
peculiar graphs and for complements of triangle-free graphs. The same authors, settled  the
case of $3$-colourable elementary graphs \cite{GM97}, and  raised the question whether all claw-free perfect graphs are chromatic choosable and settled the case of $4$-colorable elementary graphs \cite{GMP2016}. Theorem~\ref{thm:main} settles the elementary case
completely.

We borrow heavily from several previous works. First is the proof of Ohba’s conjecture \cite{NRW15,NWWZ15}. Two devices from that work is used here in modified form. The first is saturation: just as a counterexample to Ohba’s conjecture may be assumed to be complete multipartite because adding edges does not make a bad list assignment colourable \cite{NRW15,NWWZ15}, our minimal counterexample is chosen with the maximum number of edges, forcing every incomplete augment to have a rigid shape (Lemma~\ref{lem:saturation}). However, unlike the chromatic number in Ohba's, being elementary is not preserved by arbitrary edge additions. So this step leans on the Maffray--Reed characterisation \cite{MR99}. The second idea we borrow is list reduction: removing a non-adjacent pair that shares a colour and deleting that colour from the remaining lists in the deficiency lemma (Lemma~\ref{lem:deficiency}) is the two-vertex analogue of  a similar idea from \cite{NWWZ15}, and the residual lists of \eqref{eq:residual} are the standard reduction of \cite{NRW15}. What is new is the use of minimality on a smaller graph whose lists have been deliberately altered: the matched vertices of one side receive a copy of a single free vertex’s list, chosen through a surplus-minimising set, and it is this copied colouring that the deficient set of Lemma~\ref{lem:deficiency} cannot coexist with.

\subsection{Overview of the proof}\label{sec:overview}
We prove the  more flexible statement that, for every $k\ge1$, every elementary graph with clique number at most $k$ is $k$-choosable, and we argue by minimal
counterexample. Let $G$ be a counterexample with the fewest vertices and, subject to that,
the most edges. Let $L$ be an assignment of $k$-lists from which $G$ cannot be coloured. The Figure \ref{fig:flow} presents the structure of the proof.
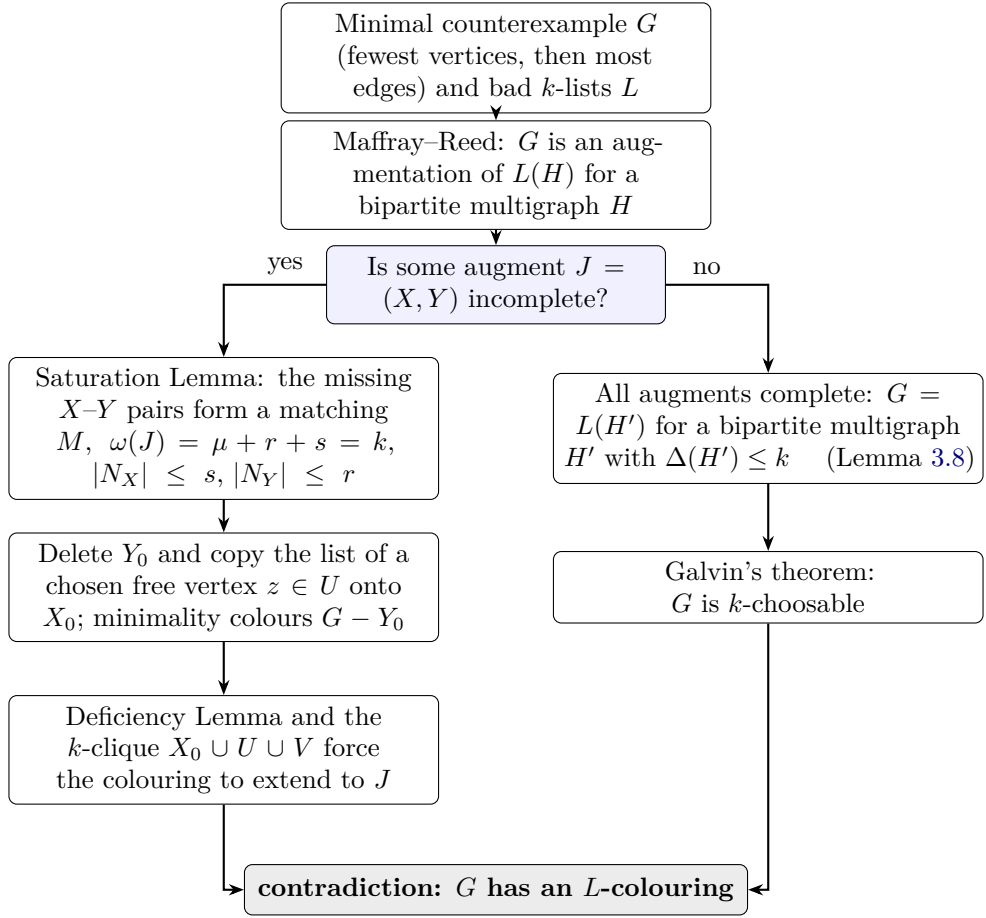
\begin{figure}[t]
\centering
\begin{tikzpicture}[x=1cm,y=1cm]
\node[box] (start) at (0,0) {Minimal counterexample $G$ (fewest vertices, then most edges) and bad $k$-lists $L$};
\node[box] (mr) at (0,-1.55) {Maffray--Reed: $G$ is an augmentation of $L(H)$ for a bipartite multigraph $H$};
\node[dec] (q) at (0,-3.0) {Is some augment $J=(X,Y)$ incomplete?};
\node[box] (sat) at (-3.6,-4.9) {Saturation Lemma: the missing $X$--$Y$ pairs form a matching $M$, $\ \omega(J)=\mu+r+s=k$, $\ |\NX|\le s$, $|\NY|\le r$};
\node[box] (red) at (-3.6,-7.0) {Delete $Y_0$ and copy the list of a chosen free vertex $z\in U$ onto $X_0$; minimality colours $G-Y_0$};
\node[box] (def) at (-3.6,-9.15) {Deficiency Lemma and the $k$-clique $X_0\cup U\cup V$ force the colouring to extend to $J$};
\node[box] (lg) at (3.6,-4.9) {All augments complete: $G=L(H')$ for a bipartite multigraph $H'$ with $\Delta(H')\le k$ \hfill (Lemma~\ref{lem:linegraph})};
\node[box] (gal) at (3.6,-7.0) {Galvin's theorem: $G$ is $k$-choosable};
\node[fin] (contra) at (0,-11.0) {contradiction: $G$ has an $L$-colouring};
\draw[arr] (start)--(mr);
\draw[arr] (mr)--(q);
\draw[arr] (q.west) -| node[pos=0.2,above,lbl]{yes} (sat.north);
\draw[arr] (q.east) -| node[pos=0.2,above,lbl]{no} (lg.north);
\draw[arr] (sat)--(red);
\draw[arr] (red)--(def);
\draw[arr] (lg)--(gal);
\draw[arr] (def.south) |- (contra.west);
\draw[arr] (gal.south) |- (contra.east);
\end{tikzpicture}
\caption{The structure of the proof.}
\label{fig:flow}
\end{figure}
By the theorem of Maffray and Reed, $G$ is obtained from the line graph of a bipartite
multigraph $H$ by \emph{augmenting} a matching of flat edges: each such edge $xy$ is replaced by
a complement of a bipartite graph, that is, by two cliques $X$ and $Y$ with some edges between
them, $X$ inheriting the neighbours of $x$ and $Y$ those of $y$.  
If every augment $(X,Y)$ is \emph{complete}---all edges between $X$ and $Y$ present---then $G$ is
itself the line graph of a bipartite multigraph, obtained from $H$ by multiplying edges, and
Galvin's theorem colours it.
So the work aims to show that
a minimal counterexample has no \emph{incomplete} augment. 
\par 
Because $G$ has the maximum number of edges among counterexamples of its order, an incomplete
augment $J=(X,Y)$ is \emph{saturated}: its missing $X$--$Y$ pairs form a matching $M$ of size
$\mu\ge1$, and its clique number $|X|+|Y|-\mu$ equals $k$ exactly. 
If $X_0,Y_0$ are the matched vertices and $U,V$  the free vertices, with $|U|=r$ and $|V|=s$,
we have $k=\mu+r+s$. The exterior neighbourhoods $\NX$ and $\NY$ of the two sides are disjoint
cliques, each complete to its side, and we have $|\NX|\le s$ and $|\NY|\le r$.
Suppose one augment is a complete graph on $X\cup Y$ minus a matching
$M$ between $X$ and $Y$, with lists of size at least $|X|$ on $X$ and at least $|Y|$ on $Y$. The Deficiency Lemma~\ref{lem:deficiency} 
says that if such a graph cannot be coloured, the failure is witnessed by a set $T$ of \emph{free}
vertices meeting both $U$ and $V$ whose lists have union of size at most $|T|+\mu-1$.  \par The proof is
a short induction: a Hall violating set either contains no whole missing pair, and is essentially
$T$, or contains a missing pair, whose two vertices share a colour and can be removed together.
We delete the $\mu$ vertices of $Y_0$, obtaining a smaller
elementary graph $G'$, and colour it from lists $L'$ that agree with $L$ except on $X_0$, where
every vertex receives a copy of the list of one carefully chosen free vertex $z\in U$. Minimality
of $G$ provides an $L'$-colouring $f$ of $G'$. We keep $f$ outside $J$ and try to recolour $J$
from the \emph{residual} lists $R$ (the original lists minus the colours that $f$ uses on $\NX$,
respectively $\NY$). The bounds $|\NX|\le s$ and $|\NY|\le r$ make $R$ satisfy the hypothesis
of the Deficiency Lemma, so if the recolouring failed we would get a deficient set $T$ of free
vertices. The vertex $z$ is chosen inside a set $S\subseteq U$ of minimum \emph{surplus}
$|L(S)|-|S|$; this guarantees that $T$ can be enlarged to $T'=T\cup S\ni z$ without losing
deficiency. But $X_0\cup U\cup V$ is a clique of size $k$ in $G'$ on which $f$ is injective, and
$f$ colours $X_0$ from $R(z)$ and $T'$ from $R(T')$: that is $\mu+|T'|$ distinct colours inside a
set of at most $\mu+|T'|-1$ colours. This contradiction shows that the recolouring succeeds, so
$G$ is $L$-colourable after all.
Why copy the list of $z$ onto $X_0$? The colouring $f$ of $G'$ is what rules out the deficient
set $T$, but only if the colours $f$ places on $X_0$ are known to lie in $R(T')$. Copying $L(z)$
achieves  this: the vertices of $X_0$ have the same exterior neighbourhood $\NX$ as $z$,
so their $f$-colours lie in $L(z)\setminus f(\NX)=R(z)\subseteq R(T')$. Deleting $Y_0$ rather than
$X_0$ is what turns $X_0\cup U\cup V$ into a clique.
\subsection{ Notation}\label{sec:notation}
For a set $S$ of vertices and a list assignment $P$ we write $P(S)=\bigcup_{v\in S}P(v)$, with
$P(\varnothing)=\varnothing$. Two disjoint vertex sets $A,B$ are \emph{complete} to each other if
every vertex of $A$ is adjacent to every vertex of $B$. $G-S$ is the graph obtained by deleting
$S$, and $N(v)$ is the neighbourhood of $v$. Table~\ref{tab:notation} collects the notation of
the main argument for reference.
\begin{table}[t]
\centering\small
\begin{tabular}{@{}>{$}l<{$}p{106mm}@{}}
\toprule
G,\ L,\ k & the minimal counterexample, its bad $k$-lists, and the clique bound (\S\ref{sec:mincex})\\
H,\ \LH & the bipartite multigraph and its line graph in the Maffray--Reed representation of $G$\\
J=(X,Y) & an augment: $X$ and $Y$ are cliques, $X\cup Y=V(J)$\\
\eX=\pX c,\ \eY=c\pY & the two edges of $H$ forming the flat edge of $\LH$ replaced by $J$; $\deg_H c=2$\\
\NX,\ \NY & neighbours of $X$, resp.\ $Y$, outside $J$ (disjoint cliques, Lemma~\ref{lem:anatomy})\\
M,\ \mu & the matching of non-adjacent $X$--$Y$ pairs in an incomplete augment, and its size\\
X_0,Y_0;\ U,V & vertices of $X$, $Y$ covered by $M$ (\emph{matched}), resp.\ not covered (\emph{free})\\
r=|U|,\ s=|V| & so that $|X|=\mu+r$, $|Y|=\mu+s$, and $k=\mu+r+s$ (Lemma~\ref{lem:saturation})\\
\sur(S)=|L(S)|-|S| & the \emph{surplus} of a nonempty set $S$ with respect to $L$\\
m_X,\ m_Y & minimum surplus of a nonempty subset of $U$, resp.\ of $V$\\
S\ni z & a subset of $U$ of surplus $m_X$ and a vertex in it\\
G'=G-Y_0,\ L',\ f & the reduced graph, its lists ($L(z)$ on $X_0$, $L$ elsewhere), an $L'$-colouring\\
\FX=f(\NX),\ \FY=f(\NY) & the colours used by $f$ on the exterior neighbourhoods\\
R & residual lists on $J$: $L(x)\setminus\FX$ on $X$, $L(y)\setminus\FY$ on $Y$\\
T,\ T_U,\ T_V,\ T' & a deficient set of free vertices, its parts in $U$ and $V$, and $T\cup S$\\
\bottomrule
\end{tabular}
\caption{Notation used in proofs.}
\label{tab:notation}
\end{table}
\section{Main tools}\label{sec:tools}
\subsection{Hall's theorem and deficient sets}
Let $P$ be a list assignment on a finite set $W$ of vertices. For a nonempty $S\subseteq W$
we call $\sur_P(S)=|P(S)|-|S|$ the \emph{surplus} of $S$, and we say that $S$ is \emph{deficient} if $\sur_P(S)<0$, that is,
if $|P(S)|\le|S|-1$. Now the Hall's theorem says the following: 
\begin{lemma}[Hall \cite{Hall35}]\label{lem:hall}
Let $P$ be a list assignment on a finite set $W$. One can choose colours $c(w)\in P(w)$,
$w\in W$, that are pairwise distinct if and only if no nonempty subset of $W$ is deficient.
\end{lemma}
Pairwise distinct colours form a proper colouring of every graph on $W$, so:
\begin{corollary}\label{cor:hall}
If a graph with vertex set $W$ has no $P$-colouring, then some nonempty $S\subseteq W$ is
deficient.
\end{corollary}
A graph is \emph{cobipartite} if its vertex set can be partitioned to two cliques; equivalently,
it is the complement of a bipartite graph.
\begin{lemma}\label{lem:konig}
Let $F$ be an $X \text{---} Y$ bipartite graph and let $J$ be its complement.  If $\nu$ is the maximum size of a matching in $F$, then
$\omega(J)=|X|+|Y|-\nu$.
\end{lemma}
\begin{proof}
Since $X$ and $Y$ are cliques, a set $K\subseteq X\cup Y$ is a clique of $J$ if and only if it
contains no edge of $F$, that is, if and only if $(X\cup Y)\setminus K$ meets every edge of $F$
(is a \emph{vertex cover} of $F$). Hence $\omega(J)=|X|+|Y|-\tau$, where $\tau$ is the minimum size
of a vertex cover of $F$. By K\H{o}nig's theorem \cite{Kon31}, $\tau=\nu$ for bipartite graphs.
\end{proof}
The following special cobipartite graphs are  very useful in our proofs.
\begin{definition}\label{def:KXYM}
Let $X$ and $Y$ be disjoint finite sets and let $M$ be a \emph{matching between $X$ and $Y$}.
We write $\Kminus{X}{Y}{M}$ 
for the graph with vertex set $X\cup Y$ in which two distinct vertices are adjacent unless they
form a pair of $M$. 
We call the vertices covered by $M$ \emph{matched} and the others \emph{free}, and we use
the notation
\begin{gather*}
X_0=\{x\in X:\ x\text{ is matched}\},\qquad Y_0=\{y\in Y:\ y\text{ is matched}\},\\
U=X\setminus X_0,\qquad V=Y\setminus Y_0,\qquad r=|U|,\qquad s=|V|,
\end{gather*}
so that $|X_0|=|Y_0|=\mu$, $|X|=\mu+r$ and $|Y|=\mu+s$.
\end{definition}
A clique of $\Kminus{X}{Y}{M}$ contains at most one vertex of each pair of $M$, and $X_0\cup U\cup V$
is a clique; hence
\begin{equation}\label{eq:omegaKXYM}
\omega(\Kminus{X}{Y}{M})=\mu+r+s .
\end{equation}
\begin{lemma}[Deficiency Lemma]\label{lem:deficiency}
Let $J=\Kminus{X}{Y}{M}$, 
and let $R$ be a list
assignment on $X\cup Y$ with $|R(x)|\ge|X|\ \ (x\in X)\qquad\text{and}\qquad |R(y)|\ge|Y|\ \ (y\in Y).$ 
If $J$ has no $R$-colouring, then there is a set $T$ of free vertices, $T\subseteq U\cup V$,
that meets both $U$ and $V$ and satisfies
\begin{equation}\label{eq:witness}
|R(T)|\le|T|+\mu-1 .
\end{equation}
\end{lemma}
The lemma essentially says that a failure to colour is always witnessed by free vertices on \emph{both} sides, and
the witness is deficient ``up to the $\mu$ missing pairs''.
\begin{proof}
We use induction on $\mu$. By Halls' theorem, there is a nonempty $W\subseteq X\cup Y$
with $|R(W)|\le|W|-1$. Such a $W$ meets both $X$ and $Y$: if, say, $W\subseteq X$, then any
$x\in W$ gives $|R(W)|\ge|R(x)|\ge|X|\ge|W|$. We have the following cases. 

\emph{Case 1: $W$ contains no pair of $M$.}
Then $W$ contains at most one vertex from each pair of $M$, so at most $\mu$ matched vertices.
Put $T=W\cap(U\cup V)$. Then $|W|\le|T|+\mu$ and
\[|R(T)|\le|R(W)|\le|W|-1\le|T|+\mu-1 .\]
We now show that $T$ meets both $U$ and $V$. If $T\cap U=\varnothing$, then $W$ consists of
at most $\mu$ matched vertices together with vertices of $V$, so $|W|\le\mu+s=|Y|$. Since $W$ meets
$Y$, we have $|R(W)|\ge|Y|\ge|W|$, a contradiction. Symmetrically, $T\cap V\ne\varnothing$.
When $\mu=0$ there are no pairs at all, and the base case holds. 

\emph{Case 2: $W$ contains both vertices $x,y$ of some pair of $M$.}
Since $R(x)\cup R(y)\subseteq R(W)$,
\[
|R(x)\cap R(y)|\ \ge\ |R(x)|+|R(y)|-|R(W)|\ \ge\ |X|+|Y|-(|W|-1)\ \ge\ 1 .
\]
Choose a colour $c\in R(x)\cap R(y)$. The graph $J'=J-\{x,y\}$ is $\Kminus{X\setminus x}{Y\setminus y}{(M\setminus\{xy\})}$.
It has $\mu-1$ missing pairs and the same free sets $U$ and $V$. Assign it the lists
$R'(w)=R(w)\setminus\{c\}$, which have at least $|X|-1=|X\setminus x|$ colours on $X\setminus x$ and at least
$|Y|-1$ colours on $Y\setminus y$. If $J'$ had an $R'$-colouring, giving $x$ and $y$ the colour $c$ would
extend it to an $R$-colouring of $J$ ($x$ and $y$ are non-adjacent, and no other vertex uses $c$).
So $J'$ has none, and by induction there is $T\subseteq U\cup V$ meeting both $U$ and $V$ with
$|R'(T)|\le|T|+(\mu-1)-1$. Restoring $c$ adds at most one colour: $|R(T)|\le|R'(T)|+1\le|T|+\mu-1$.
\end{proof}
\begin{corollary}\label{cor:deficiency}
In the following special cases of  Lemma~\ref{lem:deficiency}, $J$ has an $R$-colouring:
\begin{enumerate}[label=\textup{(\alph*)}]
\item $U=\varnothing$ or $V=\varnothing$;
\item every vertex of $U$, or every vertex of $V$, has a list of at least $\omega(J)=\mu+r+s$ colours.
\end{enumerate}
\end{corollary}
\begin{proof}
A set $T$ as in the lemma meets both $U$ and $V$, which is impossible in case (a). In case (b),
$T$ contains a vertex $w$ with $|R(w)|\ge\mu+r+s\ge\mu+|T|>|T|+\mu-1\ge|R(T)|\ge|R(w)|$, which is
absurd.
\end{proof}
\section{Main Results: Elementary graphs}\label{sec:elementary}
The \emph{line graph} $L(H)$ of a multigraph $H$ has vertex set $E(H)$, two distinct edges of
$H$ being adjacent in $L(H)$ when they share an endpoint. The edges of $H$ at a fixed vertex
form a clique of $L(H)$, the \emph{star} of that vertex. An edge of a graph is \emph{flat} if it
lies in no triangle.
\begin{definition}[Maffray--Reed \cite{MR99}]\label{def:augment}
Let $xy$ be a flat edge of a graph $G_0$, and let $B$ be a cobipartite graph, disjoint from
$G_0$, whose vertex set is partitioned into two nonempty cliques $X$ and $Y$ with at least one
edge between $X$ and $Y$. To \emph{augment} $xy$ by $B$: delete $x$ and $y$, add $B$, and make
$X$ complete to $N(x)\setminus\{y\}$ and $Y$ complete to $N(y)\setminus\{x\}$. We call $B$, or the
pair $(X,Y)$, the \emph{augment} replacing $xy$, and $X,Y$ its two \emph{sides}.
If $x_1y_1,\dots,x_hy_h$ are pairwise vertex-disjoint flat edges of $G_0$, augmenting each of them
(by pairwise disjoint augments) gives a well-defined graph: after one augmentation the remaining
edges $x_iy_i$ are still flat, and the operations commute \cite[\S1]{MR99}. Any graph so obtained,
including $G_0$ itself ($h=0$), is an \emph{augmentation} of $G_0$.
\end{definition}
Note that adjacencies between different augments are inherited from the replaced vertices:
if $x_i$ is adjacent to $x_j$ in $G_0$, then the sides $X_i$ and $X_j$ are complete to each other.
\begin{theorem}[Maffray--Reed {\cite{MR99}}]\label{thm:MR}
A graph is elementary if and only if it is an augmentation of the line graph of a bipartite
multigraph.
\end{theorem}
Both directions will be used: ``only if'' provides a representation of our counterexample, and
``if'' guarantees that modifying the edges inside an augment (Lemma~\ref{lem:saturation}) keeps
the graph elementary. Figure~\ref{fig:augment} illustrates the construction.
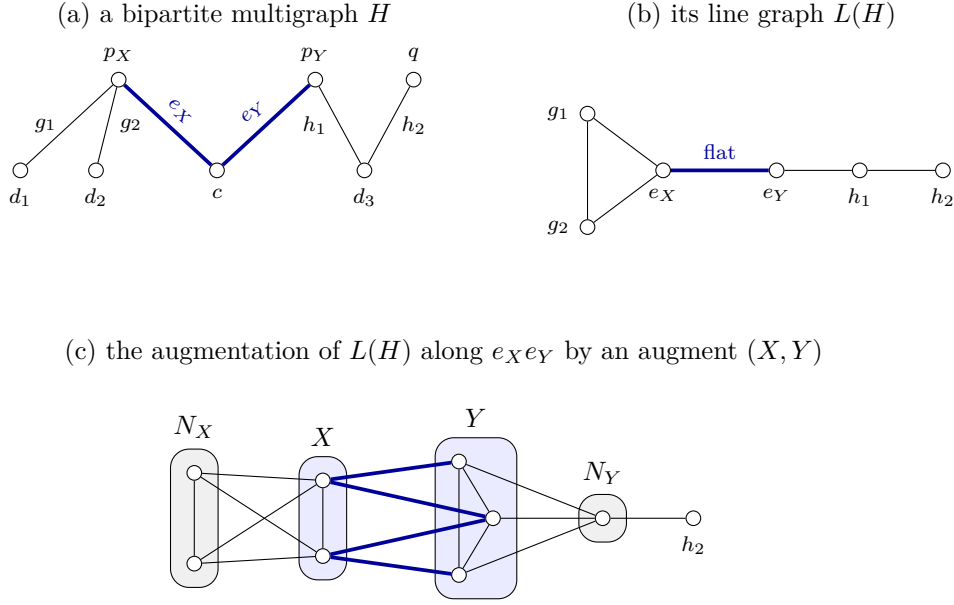
\begin{figure}[t]
\centering
\begin{tikzpicture}[x=1cm,y=1cm]
\begin{scope}
 \node[lbl] at (1.4,2.05) {(a) a bipartite multigraph $H$};
 \node[vtx,label={[slbl]above:$\pX$}] (pX) at (0,1.2) {};
 \node[vtx,label={[slbl]above:$\pY$}] (pY) at (2.6,1.2) {};
 \node[vtx,label={[slbl]above:$q$}]  (q)  at (3.9,1.2) {};
 \node[vtx,label={[slbl]below:$d_1$}] (d1) at (-1.3,0) {};
 \node[vtx,label={[slbl]below:$d_2$}] (d2) at (-0.3,0) {};
 \node[vtx,label={[slbl]below:$c$}]   (c)  at (1.3,0) {};
 \node[vtx,label={[slbl]below:$d_3$}] (d3) at (3.25,0) {};
 \draw (pX)--(d1) node[midway,left=1pt,slbl]{$g_1$};
 \draw (pX)--(d2) node[midway,right=1pt,slbl]{$g_2$};
 \draw[flat] (pX)--(c) node[midway,above,sloped,slbl]{$\eX$};
 \draw[flat] (c)--(pY) node[midway,above,sloped,slbl]{$\eY$};
 \draw (pY)--(d3) node[midway,left=1pt,slbl]{$h_1$};
 \draw (q)--(d3)  node[midway,right=1pt,slbl]{$h_2$};
\end{scope}
\begin{scope}[xshift=7.2cm]
 \node[lbl] at (1.3,2.05) {(b) its line graph $L(H)$};
 \node[vtx,label={[slbl]left:$g_1$}] (g1) at (-1.0,0.75) {};
 \node[vtx,label={[slbl]left:$g_2$}] (g2) at (-1.0,-0.75) {};
 \node[vtx,label={[slbl]below:$\eX$}] (eX) at (0,0) {};
 \node[vtx,label={[slbl]below:$\eY$}] (eY) at (1.5,0) {};
 \node[vtx,label={[slbl]below:$h_1$}] (h1) at (2.6,0) {};
 \node[vtx,label={[slbl]below:$h_2$}] (h2) at (3.7,0) {};
 \draw (g1)--(g2); \draw (g1)--(eX); \draw (g2)--(eX);
 \draw[flat] (eX)--(eY) node[midway,above,slbl]{flat};
 \draw (eY)--(h1); \draw (h1)--(h2);
\end{scope}
\begin{scope}[yshift=-4.6cm,xshift=2.3cm]
 \node[lbl] at (2.0,2.2) {(c) the augmentation of $L(H)$ along $\eX\eY$ by an augment $(X,Y)$};
 \node[vtx] (g1) at (-1.3,0.6) {}; \node[vtx] (g2) at (-1.3,-0.6) {};
 \node[vtx] (x1) at (0.4,0.5) {};  \node[vtx] (x2) at (0.4,-0.5) {};
 \node[vtx] (y1) at (2.2,0.75) {}; \node[vtx] (y2) at (2.65,0) {}; \node[vtx] (y3) at (2.2,-0.75) {};
 \node[vtx] (h1) at (4.1,0) {};    \node[vtx,label={[slbl]below:$h_2$}] (h2) at (5.3,0) {};
 \begin{scope}[on background layer]
  \node[clique,fit=(g1)(g2),label={[lbl]above:$\NX$}] {};
  \node[clique,fill=blue!8,fit=(x1)(x2),label={[lbl]above:$X$}] {};
  \node[clique,fill=blue!8,fit=(y1)(y2)(y3),label={[lbl]above:$Y$}] {};
  \node[clique,fit=(h1),label={[lbl]above:$\NY$}] {};
 \end{scope}
 \draw (g1)--(g2); \draw (g1)--(x1); \draw (g1)--(x2); \draw (g2)--(x1); \draw (g2)--(x2);
 \draw (x1)--(x2);
 \draw (y1)--(y2); \draw (y2)--(y3); \draw (y1)--(y3);
 \draw[flat] (x1)--(y1); \draw[flat] (x1)--(y2); \draw[flat] (x2)--(y2); \draw[flat] (x2)--(y3);
 \draw (h1)--(y1); \draw (h1)--(y2); \draw (h1)--(y3); \draw (h1)--(h2);
\end{scope}
\end{tikzpicture}
\caption{Augmenting a flat edge. In (a), $c$ has degree $2$, so $\eX\eY$ is a flat edge of
$L(H)$ in (b). In (c), $\eX$ is replaced by a clique $X$ and $\eY$ by a clique $Y$; $X$ inherits
the neighbours $\NX=\{g_1,g_2\}$ of $\eX$, $Y$ inherits the neighbour $\NY=\{h_1\}$ of $\eY$, and
the edges between $X$ and $Y$ (thick) are those of the augment. Every vertex of $X$ has the
same neighbourhood $\NX$ outside $J=X\cup Y$; no vertex outside $J$ sees both $X$ and~$Y$.}
\label{fig:augment}
\end{figure}
\subsection*{The structure of an augment}
We note some useful structural properties of an augmentation of the line graph of a bipartite multigraph. Fix a bipartite multigraph $H$, a set of pairwise vertex-disjoint flat edges of $\LH$, and an
augment for each of them. Call this a \emph{representation} of the resulting
augmentation $G$. For an augment $J=(X,Y)$ of the representation, let  $\NX=\bigl(\textstyle\bigcup_{x\in X}N(x)\bigr)\setminus V(J), \text{ and }
\NY=\bigl(\textstyle\bigcup_{y\in Y}N(y)\bigr)\setminus V(J)$ be the sets of neighbours of $X$,  $Y$ outside $J$ (in $G$).
\begin{lemma}\label{lem:anatomy}
Let $G$ be an augmentation of $\LH$ for a bipartite multigraph $H$. The representation can be
chosen so that every augmented flat edge $\eX\eY$ of $\LH$ consists of two non-parallel edges
$\eX=\pX c$ and $\eY=c\pY$ of $H$ whose common endpoint $c$ has degree $2$ in $H$. Moreover, for
every augment $J=(X,Y)$ of the representation:
\begin{enumerate}[label=\textup{(\alph*)}]
\item $X$ is complete to $\NX$ and $Y$ is complete to $\NY$: every vertex of $X$ has exactly the
      neighbours $\NX$ outside $J$, and every vertex of $Y$ exactly $\NY$;
\item $X\cup\NX$ and $Y\cup\NY$ are cliques of $G$;
\item $\NX\cap\NY=\varnothing$, i.e.\ no vertex outside $J$ has neighbours in both $X$ and $Y$.
      Consequently, every clique of $G$ that meets both $X$ and $Y$ is contained in $V(J)$.
\end{enumerate}
Statements \textup{(a)--(c)} remain true if the set of edges between $X$ and $Y$ is replaced by
any other nonempty set of $X$--$Y$ edges.
\end{lemma}
\begin{proof}
Let $ee'$ be a flat edge of $\LH$.  Then, $e$ and $e'$ share an endpoint. Suppose that they are
not parallel, say $e=pc$ and $e'=cp'$ with $p\ne p'$. A third edge at $c$ would be adjacent to
both $e$ and $e'$, and $pp'$ is not an edge of $H$ because $p$ and $p'$, both neighbours of $c$, lie
in the same part of the bipartition. Hence $ee'$ is flat if and only if $c$ has degree~$2$.
Another possibility is that $e,e'$ are parallel edges joining $p$ and $c$. Any other edge at $p$ or at
$c$ is adjacent to both, so $ee'$ is flat if and only if $p$ and $c$ both have degree $2$, that is,
if and only if $\{e,e'\}$ is a  component of $H$. In that case replace this component by a
path $p\,c\,p'$ with a new vertex $p'$: the line graph is unchanged. 
Fix an augmented flat edge $\eX\eY$ with $\eX=\pX c$, $\eY=c\pY$ and $\deg c=2$, and its augment
$J=(X,Y)$. In $\LH$, the neighbours of $\eX$ other than $\eY$ are exactly the other edges of $H$
at $\pX$; they form a clique $A_X$ (part of the star of $\pX$), and $A_X\cup\{\eX\}$ is a clique.
Define $A_Y$ at $\pY$ likewise. Then $A_X\cap A_Y=\varnothing$, since an edge in both would join
$\pX$ to $\pY$. \par 
Since augmentations commute, we may perform the augmentation of $\eX\eY$ last. Let $gg'$ be
another augmented flat edge. It has at most one end in $A_X$: neither end is $\eX$ (the flat edges
are vertex-disjoint), and if both $g,g'\in A_X$ then $g,g',\eX$ would be a triangle through $gg'$.
Consequently, after all the other augmentations have been performed, every vertex of $A_X$ is
either still present or has been replaced by \emph{one side} of an augment. By
Definition~\ref{def:augment}, that side is a clique complete to every neighbour of the replaced
vertex other than its partner in its own flat edge, and the partner lies outside $A_X\cup\{\eX\}$;
so the side is complete to $\eX$ and to (what has become of) the rest of $A_X$. Thus, immediately
before $\eX\eY$ is augmented, the neighbours of $\eX$ other than $\eY$ form a clique; by the rule in
Definition~\ref{def:augment}, $X$ is made complete to exactly this set, which is therefore $\NX$.
Since $X$ is a clique complete to the clique $\NX$, this proves (a) and (b) for $X$, and
symmetrically for $Y$.
For (c), a vertex of $\NX$ is either an unaugmented vertex of $A_X$ or lies in the side replacing a
vertex of $A_X$; likewise for $\NY$; and a side of an augment replaces a single vertex. Since
$A_X\cap A_Y=\varnothing$, we get $\NX\cap\NY=\varnothing$. A clique containing $x\in X$ and
$y\in Y$ cannot contain a vertex outside $J$, as such a vertex would lie in $\NX\cap\NY$.
Nothing in the above argument used the edges between $X$ and $Y$, hence any nonempty set of edges between $X,Y$ works.
\end{proof}
We now study the minimal counterexample.
\section*{The minimal counterexample}\label{sec:mincex}
We prove Theorem~\ref{thm:main} in the following form.
\begin{theorem}\label{thm:k}
For every integer $k\ge1$, every elementary graph with clique number at most $k$ is $k$-choosable.
\end{theorem}
It suffices to colour from lists of \emph{exactly} $k$ colours, since a larger list may be
replaced by any $k$ of its colours. For $k=1$ a graph with clique number at most $1$ has no edges
and is trivially $1$-choosable. So fix $k\ge2$ and suppose, for a contradiction, that some
elementary graph with clique number at most $k$ is not $k$-choosable. \emph{For the rest of the
paper}:
\begin{itemize}
    \item
$G$ is such a graph with the fewest vertices and, subject to that, the most edges.
\item $L$ is an assignment of exactly $k$ colours to each vertex of $G$ from which $G$ has no
$L$-colouring.
\end{itemize}
\begin{lemma}\label{lem:mincex}
\begin{enumerate}[label=\textup{(\roman*)}]
\item Every proper induced subgraph of $G$ has a $P$-colouring for \emph{every} list assignment
      $P$ with all lists of size $k$.
\item If $G^+$ is obtained from $G$ by adding one or more edges and $G^+$ is elementary, then
      $\omega(G^+)\ge k+1$.
\item $G$ is connected.
\end{enumerate}
\end{lemma}
 \begin{proof}
Elementary my dear Watson!  (The  lemma is immediate from the choice of $G$.)
\end{proof}
The word \emph{every} in (i) is what allows us, later, to delete a few vertices and then change
some of the remaining lists, as long as every list keeps $k$ colours.
By Theorem~\ref{thm:MR} and Lemma~\ref{lem:anatomy}, we may fix a representation of $G$ as an
augmentation of $\LH$, $H$ a bipartite multigraph, in which every augmented flat edge has the form
$\eX\eY$ with $\eX=\pX c$, $\eY=c\pY$ and $\deg_H c=2$. An augment $J=(X,Y)$ is \emph{complete} if
$X$ is complete to $Y$ (so that $J$ is a clique) and \emph{incomplete} otherwise. The next lemma
shows that the extremal choice of $G$ forces every incomplete augment to have  a very rigid structure. See Figure \ref{fig:anatomy} for an illustration.
\begin{lemma}[Saturation]\label{lem:saturation}
Let $J=(X,Y)$ be an incomplete augment in the representation of $G$. Then:
\begin{enumerate}[label=\textup{(\alph*)}]
\item the non-adjacent pairs between $X$ and $Y$ form a matching $M$ with $\mu=|M|\ge1$; thus
      $J=\Kminus{X}{Y}{M}$, and we use the notation $X_0,Y_0,U,V,r,s$ of Definition~\ref{def:KXYM}.
\item $\omega(J)=|X|+|Y|-\mu=k$, i.e.\ $k=\mu+r+s$.
\item $\NX\ne\varnothing$ and $\NY\ne\varnothing$.
\item $r\ge1$, $s\ge1$, $|\NX|\le s$ and $|\NY|\le r$.
\end{enumerate}
\end{lemma}
\begin{proof}
(a) Let $F$ be the bipartite graph with parts $X,Y$ whose edges are the non-adjacent $X$--$Y$
pairs of $J$. $F$ has at least one edge because $J$ is incomplete. Let $M$ be a maximum matching of
$F$ and $\mu=|M|\ge1$. Let $G^+$ be obtained from $G$ by adding, as edges, all pairs in
$E(F)\setminus M$. The augment $J^+=(X,Y)$ of $G^+$ is also cobipartite with at least one edge between
$X$ and $Y$. Hence $G^+$ is an augmentation of $\LH$ along the same flat edges and is elementary by
Theorem~\ref{thm:MR}. Note that its clique number is at most $k$. This is because a clique of $G^+$ that is not a clique
of $G$ contains an added pair. Thus, it meets both $X$ and $Y$, and therefore lies inside $J^+$ by
Lemma~\ref{lem:anatomy}(c). 
Moreover, $\omega(J^+)=|X|+|Y|-\mu=\omega(J)\le k$ by Lemma~\ref{lem:konig}, applied to $J^+$ (whose
non-adjacent $X$--$Y$ pairs are exactly the $\mu$ pairs of $M$) and to $J$ (where the maximum
size of a matching in $F$ is $\mu$).  Lemma~\ref{lem:mincex}(ii), then implies that no edge was added to $G$. Thus, $E(F)=M$.
(b)
As above, have $\omega(J)=|X|+|Y|-\mu\le k$.
If the inequality were strict, add a single pair of $M$ as an edge. The non-adjacent $X$--$Y$
pairs then form the matching $M$ minus one pair, so the clique number of the augment rises by
exactly one and stays at most $k$. As in the case (a), 
the new graph is elementary with clique number at
most $k$, contradicting Lemma~\ref{lem:mincex}(ii). Hence $|X|+|Y|-\mu=k$, which is
$k=\mu+r+s$ by Definition~\ref{def:KXYM}.
(c) Suppose $\NX=\varnothing$. By Lemma~\ref{lem:mincex}(i), the graph $G-V(J)$ has an
$L$-colouring $f$ (if $G=J$, take $f$ empty). Let $\FY=f(\NY)$ and define lists on $J$ by
$R(x)=L(x)$ for $x\in X$ and $R(y)=L(y)\setminus\FY$ for $y\in Y$. Since $Y\cup\NY$ is a clique
(Lemma~\ref{lem:anatomy}(b)), $|Y|+|\NY|\le k$, so $|R(y)|\ge k-|\NY|\ge|Y|$; and $|R(x)|=k\ge|X|$.
This implies $J$ admits an $R$-colouring $g$ by Corollary~\ref{cor:deficiency}.
Then $f\cup g$ is an $L$-colouring of $G$.
This contradicts the
choice of $L$. Hence $\NX\ne\varnothing$, and $\NY\ne\varnothing$ by symmetry.
(d) By (b) and (c), $X\cup\NX$ is a clique with $|X|+|\NX|\le k$ and $\NX\ne\varnothing$, so
$|X|\le k-1$ and therefore $s=k-|X|\ge1$ and $|\NX|\le k-|X|=s$. Symmetrically $r\ge1$ and
$|\NY|\le r$.
\end{proof}
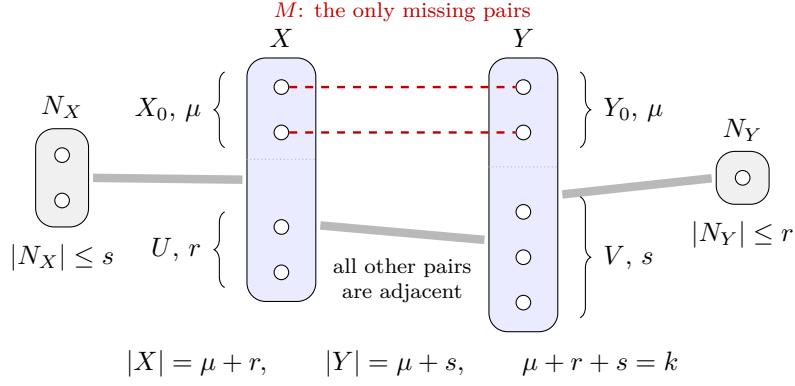
\begin{figure}[htbp]
\centering
\begin{tikzpicture}[x=1cm,y=1cm]
 
 \node[vtx] (x1) at (0,1.5) {}; \node[vtx] (x2) at (0,0.9) {};
 \node[vtx] (u1) at (0,-0.35) {}; \node[vtx] (u2) at (0,-0.95) {};
 
 \node[vtx] (y1) at (3.2,1.5) {}; \node[vtx] (y2) at (3.2,0.9) {};
 \node[vtx] (v1) at (3.2,-0.15) {}; \node[vtx] (v2) at (3.2,-0.75) {}; \node[vtx] (v3) at (3.2,-1.35) {};
 
 \node[vtx] (n1) at (-2.9,0.6) {}; \node[vtx] (n2) at (-2.9,0.0) {};
 \node[vtx] (m1) at (6.1,0.3) {};
 \begin{scope}[on background layer]
  \node[clique,fill=blue!8,inner xsep=10pt,inner ysep=8pt,fit=(x1)(u2),label={[lbl]above:$X$}] (X) {};
  \node[clique,fill=blue!8,inner xsep=10pt,inner ysep=8pt,fit=(y1)(v3),label={[lbl]above:$Y$}] (Y) {};
  \node[clique,inner sep=7pt,fit=(n1)(n2),label={[lbl]above:$\NX$}] (NX) {};
  \node[clique,inner sep=7pt,fit=(m1),label={[lbl]above:$\NY$}] (NY) {};
 \end{scope}
 
 \draw[gray!60,densely dotted] ($(X.west)+(0,0.27)$) -- ($(X.east)+(0,0.27)$);
 \draw[gray!60,densely dotted] ($(Y.west)+(0,0.37)$) -- ($(Y.east)+(0,0.37)$);
 
 \draw[complete] (NX.east) -- (X.west);
 \draw[complete] (Y.east) -- (NY.west);
 \draw[complete] ($(X.east)+(0,-0.6)$) -- ($(Y.west)+(0,-0.6)$);
 \node[slbl,align=center] at (1.6,-1.05) {all other pairs\\are adjacent};
 
 \draw[missing] (x1)--(y1); \draw[missing] (x2)--(y2);
 \node[slbl,red!75!black] at (1.6,2.5) {$M$: the only missing pairs};
 
 \draw[decorate,decoration={brace,amplitude=4pt}] (-0.75,0.7) -- (-0.75,1.7)
    node[midway,left=5pt,lbl]{$X_0$, $\mu$};
 \draw[decorate,decoration={brace,amplitude=4pt}] (-0.75,-1.15) -- (-0.75,-0.15)
    node[midway,left=5pt,lbl]{$U$, $r$};
 \draw[decorate,decoration={brace,amplitude=4pt,mirror}] (3.95,0.7) -- (3.95,1.7)
    node[midway,right=5pt,lbl]{$Y_0$, $\mu$};
 \draw[decorate,decoration={brace,amplitude=4pt,mirror}] (3.95,-1.55) -- (3.95,0.05)
    node[midway,right=5pt,lbl]{$V$, $s$};
 
 \node[lbl] at (-2.9,-0.75) {$|\NX|\le s$};
 \node[lbl] at (6.1,-0.45) {$|\NY|\le r$};
 \node[lbl] at (1.6,-2.15) {$|X|=\mu+r,\qquad |Y|=\mu+s,\qquad \mu+r+s=k$};
\end{tikzpicture}
\caption{An incomplete augment $J=(X,Y)$ of the minimal counterexample after
Lemma~\ref{lem:saturation}. Shaded regions are cliques; a thick grey band between two regions
means that they are complete to each other. The dashed pairs form the matching $M$ and are the
only non-adjacent pairs inside $X\cup Y$. There are no edges between $\NX$ and $Y$, nor between
$\NY$ and $X$ (edges between $\NX$ and $\NY$ may or may not exist).}
\label{fig:anatomy}
\end{figure}
\begin{proposition}\label{prop:noincomplete}
Every augment in the representation of $G$ is complete.
\end{proposition}
\begin{proof}
Suppose that $J=(X,Y)$ is an incomplete augment. Throughout the proof we use the notation and
the conclusions of Lemma~\ref{lem:saturation}: $J=\Kminus{X}{Y}{M}$ with $|M|=\mu\ge1$,
$X=X_0\cup U$ and $Y=Y_0\cup V$ with $|U|=r\ge1$ and $|V|=s\ge1$, $k=\mu+r+s$, and the exterior
neighbourhoods $\NX,\NY$ with $|\NX|\le s$ and $|\NY|\le r$. Recall from Lemma~\ref{lem:anatomy}
that $\NX$ and $\NY$ are disjoint, that $X\cup\NX$ and $Y\cup\NY$ are cliques, and that the pairs
of $M$ are the only non-adjacent pairs inside $X\cup Y$.
Our proof proceeds in three phases. In the first phase, we fix a side of $J$ and choose a specific subset with least surplus of colours. We choose a specific vertex and and copy its list to 
vertices in $X_0$ and try to colour $G-Y_0$ and then extend this to whole of $J$.  
\stepname{Fix a side and a set $S$.}
For a nonempty set $S$ of vertices write $\sur(S)=|L(S)|-|S|$ for its surplus with respect to the
original lists $L$, and let $m_X=\min\{\sur(S):\ \varnothing\ne S\subseteq U\}$, and 
$m_Y=\min\{\sur(S):\ \varnothing\ne S\subseteq V\}$. Due to symmetry, we may assume without loss of generality that,
\begin{equation}\label{eq:side}
m_X+r\ \le\ m_Y+s .
\end{equation}
Fix a nonempty $S\subseteq U$ with $\sur(S)=m_X$, and fix a vertex $z\in S$. We note two
consequences of this choice. First, since $|L(S)|\ge|L(z)|=k$,
\begin{equation}\label{eq:Sbig}
|S|\ \ge\ k-m_X .
\end{equation}
Second, for every nonempty $I\subseteq S$ we have $\sur(I)\ge m_X=\sur(S)$ and $L(I)\subseteq L(S)$,
so $|L(S)|-|L(I)|\le|S|-|I|$, that is,
\begin{equation}\label{eq:Smin}
|L(S)\setminus L(I)|\ \le\ |S\setminus I|\qquad\text{for every nonempty }I\subseteq S .
\end{equation}
In words, \eqref{eq:Smin} says that the vertices of $S\setminus I$ contribute at most one new
colour each beyond $L(I)$.
\stepname{Colouring $G-Y_0$ with the list of $z$ copied onto $X_0$.}
Let $G'=G-Y_0$, a proper induced subgraph of $G$ because $\mu\ge1$, and give it the lists
\[
L'(v)=\begin{cases} L(z), & v\in X_0,\\ L(v), & v\notin X_0.\end{cases}
\]
Every list $L'(v)$ has $k$ colours, so Lemma~\ref{lem:mincex}(i) provides an $L'$-colouring $f$ of
$G'$. The sets $\NX$ and $\NY$ lie outside $J$, and hence in $G'$. Set
\[
\FX=f(\NX),\qquad \FY=f(\NY),\qquad\text{then } |\FX|\le s,\ |\FY|\le r .
\]
Now define the \emph{residual lists} on $V(J)$, using the \emph{original} lists,
\begin{equation}\label{eq:residual}
R(x)=L(x)\setminus\FX\ \ (x\in X),\qquad R(y)=L(y)\setminus\FY\ \ (y\in Y).
\end{equation}
Then
\begin{equation}\label{eq:Rsizes}
|R(x)|\ \ge\ k-s\ =\ |X|\quad(x\in X),\qquad |R(y)|\ \ge\ k-r\ =\ |Y|\quad(y\in Y).
\end{equation}
The following two observations will be useful (see Figure~\ref{fig:reduced}):
\begin{enumerate}[label=\textup{(B\arabic*)},leftmargin=3em]
\item\label{B1} \emph{If $J$ has an $R$-colouring $g$, then $G$ has an $L$-colouring.}\\
Define $c=g$ on $V(J)$ and $c=f$ on $V(G)\setminus V(J)$. Lists are respected, because
$R\subseteq L$ on $J$ and $L'=L$ off $X_0\subseteq V(J)$. Edges inside $J$ are properly coloured
by $g$, edges outside $J$ by $f$. An edge from $x\in X$ to a vertex $w\notin V(J)$ has
$w\in\NX$, so $f(w)\in\FX$ while $g(x)\notin\FX$; similarly for $Y$ and $\FY$.
\item\label{B2} \emph{$Q=X_0\cup U\cup V$ is a clique of $G'$ with $|Q|=\mu+r+s=k$; moreover
$f(x)\in R(z)$ for every $x\in X_0$, and $f(w)\in R(w)$ for every $w\in U\cup V$.}\\ 
Indeed $Q\subseteq V(G')$, and every non-adjacent pair inside $X\cup Y$ is a pair of $M$, which
has a vertex in the deleted set $Y_0$; so $Q$ is a clique. Let $x\in X_0$. Then $f(x)\in L'(x)=L(z)$,
and $x$ is adjacent in $G'$ to all of $\NX$ (Lemma~\ref{lem:anatomy}(a)), so $f(x)\notin\FX$; hence
$f(x)\in L(z)\setminus\FX=R(z)$. A vertex $w\in U$ has $f(w)\in L(w)$ and is adjacent to all of $\NX$, so
$f(w)\in R(w)$; a vertex $w\in V$ is adjacent to all of $\NY$, so likewise $f(w)\in R(w)$.
\end{enumerate}
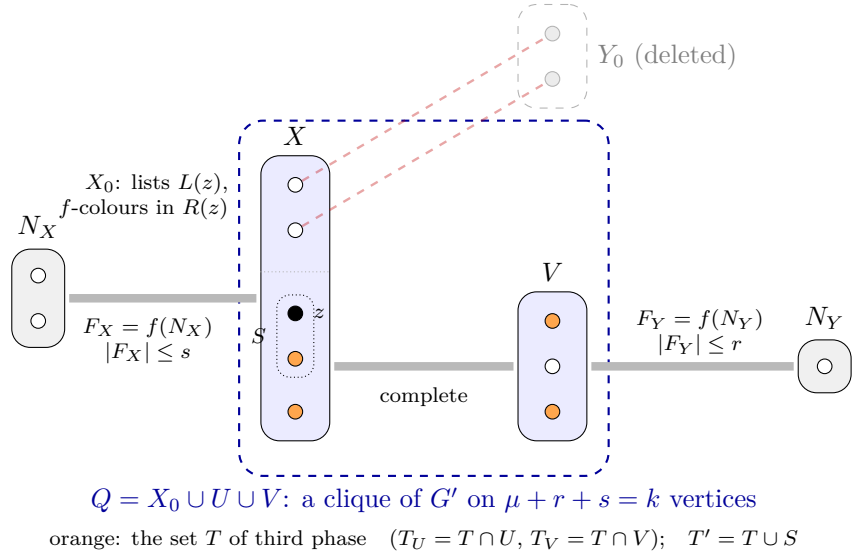
\begin{figure}[htbp]
\centering
\begin{tikzpicture}[x=1cm,y=1cm]
 
 \node[vtx] (x1) at (0,1.5) {}; \node[vtx] (x2) at (0,0.9) {};
 \node[vtxz,label={[slbl]right:$z$}] (z) at (0,-0.2) {};
 \node[vtxT] (u2) at (0,-0.8) {};
 \node[vtxT] (u3) at (0,-1.5) {};
 
 \node[vtxT] (v1) at (3.4,-0.3) {}; \node[vtx] (v2) at (3.4,-0.9) {}; \node[vtxT] (v3) at (3.4,-1.5) {};
 
 \node[vtxdel] (y1) at (3.4,3.5) {}; \node[vtxdel] (y2) at (3.4,2.9) {};
 
 \node[vtx] (n1) at (-3.4,0.3) {}; \node[vtx] (n2) at (-3.4,-0.3) {};
 \node[vtx] (m1) at (7.0,-0.9) {};
 \begin{scope}[on background layer]
  \node[clique,fill=blue!8,inner xsep=10pt,inner ysep=8pt,fit=(x1)(u3),label={[lbl]above:$X$}] (X) {};
  \node[clique,fill=blue!8,inner xsep=10pt,inner ysep=8pt,fit=(v1)(v3),label={[lbl]above:$V$}] (V) {};
  \node[draw=gray!60,dashed,rounded corners=7pt,inner xsep=10pt,inner ysep=8pt,fit=(y1)(y2),
        label={[lbl,text=gray]right:$Y_0$ (deleted)}] (Y0) {};
  \node[clique,inner sep=7pt,fit=(n1)(n2),label={[lbl]above:$\NX$}] (NX) {};
  \node[clique,inner sep=7pt,fit=(m1),label={[lbl]above:$\NY$}] (NY) {};
  \node[draw=blue!60!black,dashed,thick,rounded corners=9pt,inner xsep=8pt,inner ysep=13pt,fit=(X)(V)] (Q) {};
 \end{scope}
 \draw[gray!60,densely dotted] ($(X.west)+(0,0.35)$) -- ($(X.east)+(0,0.35)$);
 \node[draw,densely dotted,rounded corners=5pt,inner sep=4pt,fit=(z)(u2),label={[slbl]left:$S$}] {};
 
 \draw[complete] (NX.east) -- (X.west);
 \draw[complete] (V.east) -- (NY.west);
 \draw[complete] (X.east |- V.west) -- (V.west);
 
 \draw[missing,opacity=0.35] (x1)--(y1); \draw[missing,opacity=0.35] (x2)--(y2);
 
 \node[slbl,align=center] at (-1.95,-0.55) {$\FX=f(\NX)$\\$|\FX|\le s$};
 \node[slbl,align=center] at (5.35,-0.45) {$\FY=f(\NY)$\\$|\FY|\le r$};
 \node[slbl,align=right,anchor=east] at (-0.75,1.35) {$X_0$: lists $L(z)$,\\ $f$-colours in $R(z)$};
 \node[slbl,align=center] at (1.7,-1.3) {complete};
 \node[lbl,blue!60!black,anchor=north] at (Q.south) {$Q=X_0\cup U\cup V$: a clique of $G'$ on $\mu+r+s=k$ vertices};
 \node[slbl,anchor=north] at ($(Q.south)+(0,-0.55)$) {orange: the set $T$ of third phase\quad ($T_U=T\cap U$, $T_V=T\cap V$);\quad $T'=T\cup S$};
\end{tikzpicture}
\caption{The reduced instance of second phase. The matched vertices $Y_0$ are deleted; the matched
vertices $X_0$ receive the list of $z\in S\subseteq U$; the free vertices keep their lists. The
colouring $f$ of $G'$ places colours of $\FX$ on $\NX$ and of $\FY$ on $\NY$, and it is injective on
the $k$-clique $Q$. In third phase the Deficiency Lemma yields the set $T$ (orange); the choice of $S$
guarantees $T\cap S\ne\varnothing$, so $T'=T\cup S$ is still deficient, yet $f$ colours the
$\mu+|T'|$ vertices of $X_0\cup T'$ from $R(T')$.}
\label{fig:reduced}
\end{figure}
\stepname{Extend the colouring $f$ to $J$.}
We claim that $J$ has an $R$-colouring. Suppose not. By~\eqref{eq:Rsizes} the Deficiency
Lemma~\ref{lem:deficiency} applies to $J=\Kminus{X}{Y}{M}$ with the lists $R$. This gives a set
$T\subseteq U\cup V$ with
\[
T_U:=T\cap U\ne\varnothing,\qquad T_V:=T\cap V\ne\varnothing,\qquad |R(T)|\le|T|+\mu-1 .
\]
The following holds:
\begin{enumerate}[label=\textup{(C\arabic*)},leftmargin=3em]
\item\label{C1} \emph{$T$ has nonempty intersection with $S$.}\\ 
For $y\in T_V$ we have $L(y)\subseteq R(y)\cup\FY$, hence $L(T_V)\subseteq R(T)\cup\FY$. Using the
definition of $m_Y$ (note $T_V\ne\varnothing$), followed by $|\FY|\le r$ and the bound on $|R(T)|$ we have,
\[
m_Y+|T_V|\ \le\ |L(T_V)|\ \le\ |R(T)|+|\FY|\ \le\ |T_U|+|T_V|+\mu-1+r .
\]
Since $\mu+r=k-s$, this gives
\begin{equation}\label{eq:TUlower}
|T_U|\ \ge\ m_Y+s-k+1 .
\end{equation}
If $T\cap S=\varnothing$, then $T_U\subseteq U\setminus S$, and by \eqref{eq:Sbig} and \eqref{eq:side}
\[
|T_U|\ \le\ r-|S|\ \le\ m_X+r-k\ \le\ m_Y+s-k ,
\]
contradicting \eqref{eq:TUlower}. Hence $T\cap S\ne\varnothing$.
\item\label{C2} \emph{The set $T'=T\cup S$ satisfies $|R(T')|\le|T'|+\mu-1$.}\\ 
Put $I=T\cap S$. This is nonempty by \ref{C1}. All vertices of $S$ lie in $X$, so by
\eqref{eq:residual} the same set $\FX$ is removed from each of their lists, and therefore
\[
R(S)\setminus R(I)=\bigl(L(S)\setminus L(I)\bigr)\setminus\FX .
\qquad\text{Hence}\qquad
|R(S)\setminus R(I)|\le|S\setminus I|=|S\setminus T|
\]
by \eqref{eq:Smin}. Since $R(I)\subseteq R(T)$, the new colours contributed by $S$ satisfy
$R(T')\setminus R(T)=R(S)\setminus R(T)\subseteq R(S)\setminus R(I)$. Hence
\[
|R(T')|\ \le\ |R(T)|+|S\setminus T|\ \le\ |T|+\mu-1+|S\setminus T|\ =\ |T'|+\mu-1 .
\]
\end{enumerate}
Note that $X_0$ and $T'$ are disjoint subsets of the clique $Q$ of \ref{B2}, so $f$ takes $\mu+|T'|$
pairwise distinct values on $X_0\cup T'$. By \ref{B2} all of them lie in $R(z)\cup R(T')=R(T')$,
because $z\in S\subseteq T'$. Hence $|R(T')|\ge\mu+|T'|$, contradicting \ref{C2}.
This means that  $J$ must admit an $R$-colouring, and by observation \ref{B1} the graph $G$ has an $L$-colouring. This contradicts the choice of $(G,L)$ and proves the proposition.
\end{proof}
The remaining case is that  every augment of $G$ is a clique. We claim that such an augment can be absorbed into the multigraph by multiplying edges (Figure~\ref{fig:parallel}). If that is the case, Galvin's theorem that  line graph of bipartite multigraphs are chromatic choosable can be applied and we are done. Our next lemma establishes this claim.
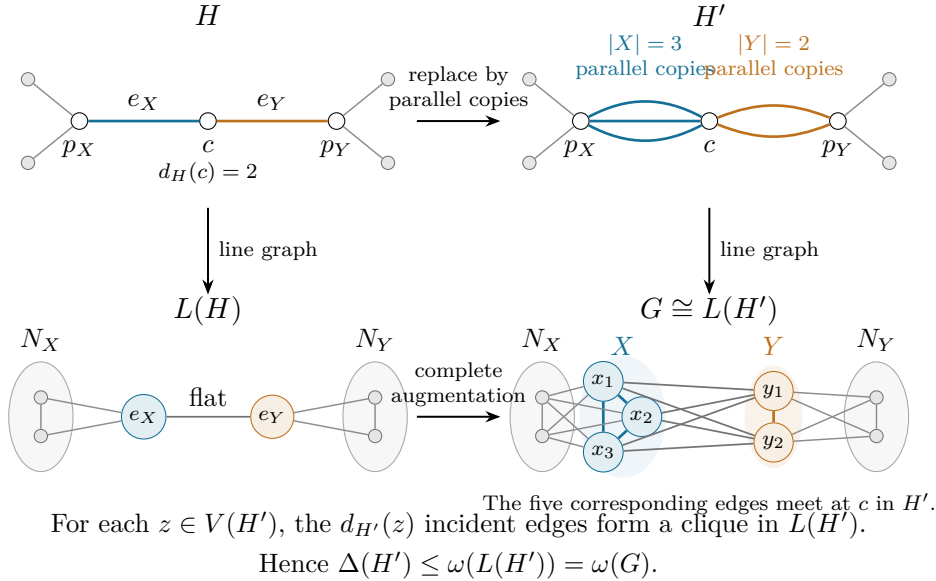
\begin{figure}[h]
\centering
\begin{tikzpicture}[scale=0.85,
  x=1cm,y=1cm,font=\small,
  root/.style={
    circle,draw,fill=white,inner sep=2.2pt
  },
  outside/.style={
    circle,draw=black!55,fill=black!10,
    inner sep=0pt,minimum size=5pt
  },
  vx/.style={
    circle,draw=Xcolour,fill=Xcolour!12,
    inner sep=1.8pt,font=\scriptsize
  },
  vy/.style={
    circle,draw=Ycolour,fill=Ycolour!12,
    inner sep=1.8pt,font=\scriptsize
  },
  unchanged/.style={draw=black!45,line width=.55pt},
  ex/.style={draw=Xcolour,line width=1pt},
  ey/.style={draw=Ycolour,line width=1pt},
  inside/.style={draw=black!55,line width=.65pt},
  map/.style={-{Stealth[length=2mm]},line width=.75pt},
  title/.style={font=\normalsize},
  group/.style={
    ellipse,draw=black!35,fill=black!3,
    minimum width=.85cm,minimum height=1.45cm
  }
]
\begin{scope}
  \node[title] at (0,1.65) {$H$};
  \node[root,label=below:$p_X$] (pX) at (-2,0) {};
  \node[root,label=below:$c$]   (c)  at (0,0) {};
  \node[root,label=below:$p_Y$] (pY) at (2,0) {};
  \draw[ex] (pX)--node[above] {$e_X$}(c);
  \draw[ey] (c)--node[above] {$e_Y$}(pY);
  \foreach \s in {-1,1}{
    \node[outside] (l\s) at (-2.8,\s*.65) {};
    \node[outside] (r\s) at (2.8,\s*.65) {};
    \draw[unchanged] (l\s)--(pX) (pY)--(r\s);
  }
  \node[font=\scriptsize] at (0,-.8) {$d_H(c)=2$};
\end{scope}
\begin{scope}[shift={(7.8,0)}]
  \node[title] at (0,1.65) {$H'$};
  \node[root,label=below:$p_X$] (qX) at (-2,0) {};
  \node[root,label=below:$c$]   (d)  at (0,0) {};
  \node[root,label=below:$p_Y$] (qY) at (2,0) {};
  
  \draw[ex] (qX) to[bend left=28] (d);
  \draw[ex] (qX)--(d);
  \draw[ex] (qX) to[bend right=28] (d);
  
  \draw[ey] (d) to[bend left=22] (qY);
  \draw[ey] (d) to[bend right=22] (qY);
  \node[text=Xcolour,align=center,font=\scriptsize]
    at (-1,1) {$|X|=3$\\parallel copies};
  \node[text=Ycolour,align=center,font=\scriptsize]
    at (1,1) {$|Y|=2$\\parallel copies};
  \foreach \s in {-1,1}{
    \node[outside] (ll\s) at (-2.8,\s*.65) {};
    \node[outside] (rr\s) at (2.8,\s*.65) {};
    \draw[unchanged] (ll\s)--(qX) (qY)--(rr\s);
  }
\end{scope}
\begin{scope}[shift={(0,-4.6)}]
  \node[title] at (0,1.65) {$L(H)$};
  \begin{scope}[on background layer]
    \node[group,label=above:$N_X$] at (-2.6,0) {};
    \node[group,label=above:$N_Y$] at (2.6,0) {};
  \end{scope}
  \node[vx] (EX) at (-1,0) {$e_X$};
  \node[vy] (EY) at (1,0) {$e_Y$};
  \draw[inside] (EX)--node[above] {flat}(EY);
  \foreach \s in {-1,1}{
    \node[outside] (A\s) at (-2.6,\s*.3) {};
    \node[outside] (B\s) at (2.6,\s*.3) {};
    \draw[unchanged] (A\s)--(EX) (EY)--(B\s);
  }
  \draw[unchanged] (A-1)--(A1) (B-1)--(B1);
\end{scope}
\begin{scope}[shift={(7.8,-4.6)}]
  \node[title] at (0,1.65) {$G\cong L(H')$};
  \begin{scope}[on background layer]
    \node[group,label=above:$N_X$] at (-2.6,0) {};
    \node[group,label=above:$N_Y$] at (2.6,0) {};
    \fill[Xcolour!7] (-1.35,0) ellipse (.67cm and .95cm);
    \fill[Ycolour!9] (1,0) ellipse (.45cm and .8cm);
  \end{scope}
  \node[text=Xcolour] at (-1.35,1.1) {$X$};
  \node[text=Ycolour] at (1,1.1) {$Y$};
  \node[vx] (x1) at (-1.65,.55) {$x_1$};
  \node[vx] (x2) at (-1.05,0) {$x_2$};
  \node[vx] (x3) at (-1.65,-.55) {$x_3$};
  \node[vy] (y1) at (1,.4) {$y_1$};
  \node[vy] (y2) at (1,-.4) {$y_2$};
  
  \draw[ex] (x1)--(x2)--(x3)--(x1);
  \draw[ey] (y1)--(y2);
  
  \foreach \i in {1,2,3}
    \foreach \j in {1,2}
      \draw[inside] (x\i)--(y\j);
  
  \foreach \s in {-1,1}{
    \node[outside] (AA\s) at (-2.6,\s*.3) {};
    \node[outside] (BB\s) at (2.6,\s*.3) {};
    \foreach \i in {1,2,3}
      \draw[unchanged] (AA\s)--(x\i);
    \foreach \j in {1,2}
      \draw[unchanged] (BB\s)--(y\j);
  }
  \draw[unchanged] (AA-1)--(AA1) (BB-1)--(BB1);
  \node[font=\scriptsize] at (0,-1.35)
    {The five corresponding edges meet at $c$ in $H'$.};
\end{scope}
\draw[map] (3.25,0)
  --node[above,align=center,font=\scriptsize]
    {replace by\\parallel copies}(4.55,0);
\draw[map] (3.25,-4.6)
  --node[above,align=center,font=\scriptsize]
    {complete\\augmentation}(4.55,-4.6);
\draw[map] (0,-1.35)
  --node[right,font=\scriptsize]{line graph}(0,-2.7);
\draw[map] (7.8,-1.35)
  --node[right,font=\scriptsize]{line graph}(7.8,-2.7);
\node[align=center] at (3.9,-6.6) {
  For each $z\in V(H')$, the $d_{H'}(z)$ incident edges
  form a clique in $L(H')$.\\[3pt]
  Hence $\Delta(H')\le\omega(L(H'))=\omega(G)$.
};
\end{tikzpicture}
\caption{Absorbing a complete augment into the multigraph (compare Figure~\ref{fig:augment}).
The $|X|$ copies of $\eX$ and the $|Y|$ copies of $\eY$ all meet at $c$, so they form the clique
$X\cup Y$ in the line graph, and they are attached at $\pX$ and $\pY$ exactly as $\eX$ and $\eY$
were, which reproduces $\NX$ and $\NY$.}
\label{fig:parallel}
\end{figure}
\begin{lemma}\label{lem:linegraph}
Suppose that every augment in the representation of $G$ is complete. Let $H'$ be obtained from
$H$ by replacing, for every augmented flat edge $\eX\eY$ with augment $(X,Y)$, the edge
$\eX=\pX c$ by $|X|$ parallel copies and the edge $\eY=c\pY$ by $|Y|$ parallel copies. Then $H'$ is
a bipartite multigraph, $G\cong L(H')$, and $\Delta(H')\le\omega(G)$.
\end{lemma}
\begin{proof}
Distinct augmented flat edges are vertex-disjoint in $\LH$, so they involve distinct edges of
$H$ and the replacements do not interfere. $H'$ has the same vertex set and bipartition as $H$,
so it is bipartite. Define $\varphi\colon V(G)\to E(H')$ by sending every vertex of $G$ that is an
unaugmented edge of $H$ to itself, the $|X|$ vertices of each side $X$ bijectively to the copies
of $\eX$, and the $|Y|$ vertices of each side $Y$ bijectively to the copies of $\eY$. It is not difficult to verify that this is a
bijection, and  that it preserves adjacency and non-adjacency. That is, $\varphi$ defines an
isomorphism $G\to L(H')$.
Finally, the edges of $H'$ at any vertex form a clique of $L(H')\cong G$, so $\Delta(H')\le\omega(G)$.
\end{proof}
\begin{proof}[Proof of Theorem~\textup{\ref{thm:k}}]
Let $G$, $L$ and $k$ be as fixed before. By Proposition~\ref{prop:noincomplete}
every augment of $G$ is complete, so by Lemma~\ref{lem:linegraph} $G\cong L(H')$ for a bipartite
multigraph $H'$ with $\Delta(H')\le\omega(G)\le k$. 
Galvin's theorem now gives an $L$-colouring of $G$. This contradicts the choice of $L$. Hence no counterexample exists, and
every elementary graph with clique number at most $k$ is $k$-choosable proving the claim and  Theorem~\ref{thm:main}.
\end{proof}


\begin{thebibliography}{99}
\setlength{\itemsep}{2pt}
\bibitem{CD19}
M.~Chudnovsky and C.~Dibek, \emph{Strongly perfect claw-free graphs---{A} short proof}, Journal of Graph Theory \textbf{97} (2021), no.~3, 359--381.

\bibitem{CS88}
V.~Chv\'atal and N.~Sbihi,
\emph{Recognizing claw-free perfect graphs},
J.~Combin.\ Theory Ser.~B 44 (1988), 154--176.

\bibitem{EOOS02}
T.~Enomoto, K.~Ohba, K.~Ota and J.~Sakamoto,
\emph{Choice number of some complete multi-partite graphs},
Discrete Math.\ 244 (2002), 55--66

\bibitem{ERT80}
P.~Erd\H{o}s, A.~L. Rubin and H.~Taylor,
\emph{Choosability in graphs},
Congr.\ Numer.\ 26 (1980), 125--157.

\bibitem{Gal95}
F.~Galvin,
\emph{The list chromatic index of a bipartite multigraph},
J.~Combin.\ Theory Ser.~B 63 (1995), 153--158,
\href{https://doi.org/10.1006/jctb.1995.1011}{doi:10.1006/jctb.1995.1011}.

\bibitem{GM97}
S.~Gravier and F.~Maffray,
\emph{The choice number of 3-colorable elementary graphs},
Discrete Math.\ 165/166 (1997), 353--358.

\bibitem{GM98}
S.~Gravier and F.~Maffray,
\emph{Graphs whose choice number is equal to their chromatic number},
J.~Graph Theory 27 (1998), 87--97.


\bibitem{GMP2016}
S.~Gravier, F.~Maffray and L.~Pastor
\emph{On the choosability of claw-free perfect graphs},
Graphs and Combinatorics(2016), 2393--2413

\bibitem{Hall35}
P.~Hall,
\emph{On representatives of subsets},
J.~London Math.\ Soc.\ 10 (1935), 26--30.

\bibitem{Kie00}
H.~A.~Kierstead,
\emph{On the choosability of complete multipartite graphs with part size three},
Discrete Math.\ 211 (2000), 255--259.

\bibitem{Kon31}
D.~K\H{o}nig,
\emph{Gr\'afok \'es m\'atrixok},
Mat.\ Fiz.\ Lapok 38 (1931), 116--119.

\bibitem{MR99}
F.~Maffray and B.~Reed,
\emph{A description of claw-free perfect graphs},
J.~Combin.\ Theory Ser.~B 75 (1999), 134--156,
\href{https://doi.org/10.1006/jctb.1998.1872}{doi:10.1006/jctb.1998.1872}.

\bibitem{NRW15}
J.~A.~Noel, B.~A.~Reed and H.~Wu,
\emph{A proof of a conjecture of Ohba},
J.~Graph Theory 79 (2015), no.~2, 86--102,
\href{https://doi.org/10.1002/jgt.21819}{doi:10.1002/jgt.21819}.

\bibitem{NWWZ15}
J.~A.~Noel, D.~B.~West, H.~Wu and X.~Zhu,
\emph{Beyond Ohba's conjecture: a bound on the choice number of $k$-chromatic graphs with $n$ vertices},
European J.\ Combin.\ 43 (2015), 295--305,
\href{https://doi.org/10.1016/j.ejc.2014.08.032}{doi:10.1016/j.ejc.2014.08.032}.

\bibitem{Ohb02}
K.~Ohba,
\emph{On chromatic-choosable graphs},
J.~Graph Theory 40 (2002), 130--135.
\bibitem{NSN}
Nandana K. Vasudevan, K. Somasundaram, and N. Narayanan,
\emph{List-Coloring and Chromatic-Choosability -- A Dynamic Survey},
arXiv:2606.31702 [math.CO], 2026, \url{[https://arxiv.org/abs/2606.31702](https://arxiv.org/abs/2606.31702)}.

\bibitem{Viz76}
V.~G. Vizing,
\emph{Coloring the vertices of a graph in prescribed colors} (Russian),
Diskret.\ Analiz 29 (1976), 3--10.

\end{thebibliography}
\end{document}